\documentclass[12pt,oneside,british,a4wide]{amsart}
\usepackage{amsmath, amssymb,verbatim,appendix}
\usepackage[mathscr]{eucal}
\usepackage{amscd}
\usepackage{amsthm}
\usepackage{stmaryrd}
\usepackage{enumerate}
\usepackage{comment}
\usepackage[latin1]{inputenc}
\usepackage{tikz}
\usetikzlibrary{shapes,arrows}
\usepackage{cite}
\usepackage{url}

\usepackage{changes}
\usepackage{setspace,a4wide,color,xcolor,graphicx}

\def\showauthornotes{1}

\ifnum\showauthornotes=1
\newcommand{\Authornote}[2]{{\sf\small\color{red}{[#1: #2]}}}
\else
\newcommand{\Authornote}[2]{}
\fi

\newtheorem{theorem}{Theorem}
\newtheorem{lemma}{Lemma}
\newtheorem{corollary}{Corollary}
\newtheorem{proposition}{Proposition}

\theoremstyle{definition}
\newtheorem{example}{Example}
\newtheorem{definition}{Definition}

\def\E{\mathbb{E}}

\def\T{\mathbb{T}}
\def\C{\mathbb{C}}

\def\F{\mathbb{F}}

\def\e{\epsilon}
\def\a{\alpha}

\newcommand{\abar}{\bar{a}}
\newcommand{\bbar}{\bar{b}}

\newcommand{\xbar}{\bar{x}}
\newcommand{\ybar}{\bar{y}}

\newcommand{\calL}{\mathcal{L}}

\newcommand{\calM}{\mathcal{M}}

\newcommand{\epboundtrivregpairs}{(1/k)^3}

\newcommand{\epindregtoy}{\epsilon^{1/2}}

\newcommand{\epboundfact}{(1/4)^{D^2}}

\newcommand{\epgoodlemtoy}[2]{4{#1}^{1/2{#2}}}

\newcommand{\epmainthmff}[2]{4{#1}^{1/2{#2}}}

\newcommand{\epkeylemtoy}[2]{4{#1}^{1/2{#2}}}
\newcommand{\boundkeylemtoy}[1]{(1/{#1})^6}
\newcommand{\boundgoodlemtoy}[1]{(1/{#1})^6}

\newcommand{\eptrivregpairs}[2]{4{#1}^{1/{#2}}}

\newenvironment{proofof}[1]{\indent{\itshape Proof of #1}:~~}{\qed}

\begin{document}

\title[Stable arithmetic regularity in the finite-field model]{Stable arithmetic regularity in the finite-field model}

\author{C. Terry}

\author{J. Wolf}

\address{Department of Mathematics, University of Maryland, College Park,
MD 20742, USA}

\email{cterry@umd.edu}

\address{School of Mathematics, University of Bristol, Bristol BS8 1TW, UK}

\email{julia.wolf@bristol.ac.uk}

\date{}

\begin{abstract}
The arithmetic regularity lemma for $\mathbb{F}_p^n$, proved by Green in 2005, states that given a subset $A\subseteq \F_p^n$, there exists a subspace $H\leq \F_p^n$ of bounded codimension such that $A$ is Fourier-uniform with respect to almost all cosets of $H$. It is known that in general, the growth of the codimension of $H$ is required to be of tower type depending on the degree of uniformity, and that one must allow for a small number of non-uniform cosets.  

Our main result is that, under a natural model-theoretic assumption of stability, the tower-type bound and non-uniform cosets in the arithmetic regularity lemma are not necessary.  Specifically, we prove an arithmetic regularity lemma for $k$-stable subsets $A\subseteq \mathbb{F}_p^n$ in which the bound on the codimension of the subspace is a polynomial (depending on $k$) in the degree of uniformity, and in which there are no non-uniform cosets. This result is an arithmetic analogue of the stable graph regularity lemma proved by Malliaris and Shelah. \end{abstract}

\maketitle


\section{Introduction}\label{sec:intro}

An important theme within model theory is the search for structural dichotomies among first-order theories.  Such dichotomies are called ``dividing lines.''  One of the earliest and most useful dividing lines is the notion of model-theoretic stability.  In particular, a theory is stable if no formula has the so-called \emph{order property}.  Stable theories have been fundamental to model theory since the early work of Shelah \cite{Shelah:1990ti}, who demonstrated that, in the setting of infinite structures, the stable/unstable diving line corresponds to the presence/absence of certain important global properties.

A finitary manifestation of this phenomenon was established by Malliaris and Shelah \cite{Malliaris:2014go}, who proved a version of Szemer\'edi's regularity lemma for stable graphs. Informally speaking, a graph is $k$-stable if it does not contain any half-graphs of height $k$. These half-graphs are well known \cite[Section 1.1]{Conlon:2012es} to be the reason why regular partitions obtained from Szemer\'edi's regularity lemma need to allow for the existence of irregular pairs. 

Malliaris and Shelah found that by excluding the presence of half-graphs of size larger than $k$, they could not only rule out the existence of irregular pairs, but reduce the bound on the number of parts in the regular partition to a function that is a polynomial (depending on $k$) in the degree of regularity. For comparison, Gowers \cite{Gowers:1997fh} had shown that in general this bound is of tower type. Moreover, in the case of stable graphs the density of the induced subgraphs between any two parts of the partition can be guaranteed to be arbitrarily close to 0 or 1, yielding additional structural information about the graph.  Thus Malliaris and Shelah proved that global structural properties of finite graphs (namely strong regularity lemmas) can be derived from the local combinatorial property of omitting half-graphs above a certain height.

In the wake of countless successful applications of Szemer\'edi's regularity lemma to problems across mathematics and theoretical computer science, a first ``arithmetic" regularity lemma was proved by Green \cite{Green:2005kh} in 2005. In its simplest form, when $G=\F_p^n$ for a small fixed prime $p$, it states that given any subset $A\subseteq \F_p^n$, there exists a subspace $H$ of bounded codimension such that $A$ is uniform with respect to almost all cosets of $H$, in the sense that its restricted indicator function has vanishingly small Fourier transform. Just like in the graph regularity lemma, it was shown that the growth of the codimension of $H$ was in general required to be of tower type (see \cite{Green:2005kh}, with a slight improvement in \cite{Hosseini:2014tv}).   

In \cite{Green:2005kh}, the arithmetic regularity lemma was used to prove a so-called ``removal lemma" for finite abelian groups, which in turn has seen numerous applications.\footnote{In fact, it turned out that the arithmetic removal lemma could be deduced directly from the graph regularity lemma, see \cite{Kral:2008wg}.} Not long afterwards, a number of more general regularity-type statements were established in which the requirement of Fourier uniformity was replaced by suitable higher-order equivalents in terms of the Gowers norms, see for example \cite{Gowers:2011et,Green:2010fk}.

In view of the above, it is natural to ask what kinds of sets permit efficient arithmetic regularity lemmas.  Addressing this question, the main result of this paper is an efficient arithmetic regularity lemma for stable subsets of finite vector spaces over fields of a fixed prime order. We shall say, rather informally for now, that a subset $A$ of a finite abelian group $G$ is $k$-stable if there do not exist sequences $a_1,\ldots, a_k,b_1,\ldots, b_k\in G$ such that $a_i+b_j\in A$ if and only if $i\leq j$. 

\begin{theorem}[Main result]\label{thm:mainff}
For all $\epsilon\in (0,1)$, $k\geq 2$, and primes $p$, there is $n_0 = n_0(k,\epsilon,p)$ such that the following holds for all $n\geq n_0$. Suppose that $G:=\mathbb{F}_p^n$ and that $A\subseteq G$ is $k$-stable. Then there is a subspace $H\leqslant G$ of codimension at most $O_k(\epsilon^{-O_k(1)})$ such that for any $g\in G$, either $|(A-g)\cap H|\leq \epsilon |H|$ or $|H\setminus (A-g)|\leq \epsilon|H|$.
\end{theorem}

In particular, this statement implies that the set $A$ is Fourier-uniform with respect to all cosets of $H$. 

Theorem \ref{thm:mainff} has all the features of the Malliaris-Shelah regularity result for graphs: the bounds are no longer of tower-type, there are no non-uniform translates, and in fact, the density of $A$ on each translate is either close to 1 or close to 0.

It is possible to deduce a much stronger structural result than in the case of graphs, however.  For instance, in Section \ref{sec:remarks} we deduce from Theorem \ref{thm:mainff} that a $k$-stable subset $A\subseteq \mathbb{F}_p^n$ must look approximately like a union of cosets of a subspace of bounded codimension.  

\begin{corollary}[Stable sets look like a union of cosets]\label{cor:ff}
For all $\epsilon\in (0,1)$, $k\geq 2$, and primes $p$, there is $n_1=n_1(k,\epsilon,p)$ and a polynomial $h(x)$ depending only on $k$ such that for each $n\geq n_1$ the following holds. Suppose that $G:=\mathbb{F}_p^n$ and that $A\subseteq G$ is $k$-stable. Then there is a subspace $H\leq G$ of codimension at most $h(1/\epsilon)$ and a set $I\subseteq G/H$ such that $|A\Delta \bigcup_{g+H\in I} (g+H)|\leq \epsilon |G|$.
\end{corollary}

We also show that a set that has small symmetric difference with a stable set enjoys a comparable regularity property to the stable set itself. To further emphasize how critical the notion of stability is to the understanding of regularity in the arithmetic context, we quantify the size of the order property in a construction by Green and Sanders \cite{Green:2015wp} which proved the necessity of a non-uniform coset.

We conclude the paper with a few further remarks and open problems. For example, we explain how Theorem \ref{thm:mainff} implies the existence of a regular partition of the Cayley graph $\Gamma(G,A)$ with the main features of the stable graph regularity lemma of Malliaris and Shelah, and the additional property that the parts in the regular partition are the cosets of a subspace. 


The structure of the remainder of the paper is as follows. We begin by introducing the model-theoretic notion of stability in more detail, by providing examples and proving some basic properties of stable sets in Section \ref{sec:stability}. We discuss the notion of regularity and the stronger variants we shall be using in Section \ref{sec:regularity}. The proof of Theorem \ref{thm:mainff}, which proceeds along the lines of Malliaris and Shelah, but with added arithmetic structural information, is presented in Section \ref{sec:proof}.  

We perform the argument in the simple setting of vector spaces over finite fields, which is an important toy model in arithmetic combinatorics \cite{Green:2005ic,Wolf:2014vy}. The plentiful presence of subgroups makes this a particularly pleasant setting to work in, and provides a blueprint for the case of general finite abelian groups which we intend to address in future work. There one has to make do with so-called approximate subgroups (or ``Bohr sets"), requiring non-trivial adaptations of the arguments in this paper.

After the first version of this paper was posted to the arXiv, the first author, Conant, and Pillay \cite{Conant:2017uw} used model-theoretic arguments to generalize some of the main results to arbitrary (large) finite groups. Specifically, they used an ultra-product construction along with deep tools from stable group theory to prove Theorem \ref{thm:mainff} and Corollary \ref{cor:ff} for stable subsets of arbitrary finite abelian groups, but without obtaining an explicit bound on the index of the subgroup. 
\vspace{8pt}

\textbf{Acknowledgements.} The authors wish to thank Terence Tao for his encouragement, and are grateful to Jacob Fox and Noga Alon for pointing out additional references. They are indebted to the Simons Institute for the Theory of Computing at UC Berkeley, where this project was conceived, for providing outstanding working conditions during the \emph{Pseudorandomness} programme in Spring 2017, and the London Mathematical Society for supporting the collaboration through a Scheme 4 grant (no. 41675).

\section{Stability}\label{sec:stability}

In this section we give some background on stability.  For completeness, we begin with the definition of stability in the usual model-theoretic context, although this will not be used directly in this paper. Suppose that $\calL$ is a first-order language, $T$ is a complete first-order theory, and $\xbar=(x_1,\ldots, x_m)$, $\ybar=(y_1,\ldots, y_n)$ are finite tuples of variables.  Then given an $\calL$-formula $\varphi(\xbar;\ybar)$ and an integer $k\geq 1$, we say that $\varphi(\xbar;\ybar)$ has the \emph{$k$-order property with respect to $T$} if there is a model $\calM\models T$, $\abar_1,\ldots, \abar_k\in M^m$, and $\bbar_1,\ldots, \bbar_k\in M^n$ such that $\calM\models \varphi(\abar_i,\bbar_j)$ if and only if $i\leq j$.  We say that $\varphi$ has the \emph{order property with respect to $T$} if it has the $k$-order property for all integers $k\geq 1$.  Then $T$ is \emph{stable} if no formula has the order property with respect to $T$.  

Stability is an example of what is called a \emph{combinatorial dividing line}, that is, a combinatorial property of formulas with interesting global consequences. 
It has strong structural implications for a theory $T$, and there is a vast literature exploring this phenomenon. To the reader unfamiliar with stability we recommend \cite{Baldwin:2017wm,Pillay:oRBBGvRc} as a starting point. The website \url{http://forkinganddividing.com} is a handy guide for exploring this area of research in more detail.  

The present paper fits into an emerging trend which considers the implications of dividing lines in the setting of \emph{finite} structures (as opposed to infinite structures where they are usually studied).  Results to date have mostly focused on showing that certain local combinatorial restrictions on graphs (motivated by model theory) imply improved versions of Szemer\'edi's regularity lemma (which is a statement about the global structure of the graph). One of the first results along these lines is due to Malliaris and Shelah \cite{Malliaris:2014go}. In order to state it we need the following definition. For convenience, if $x\neq y$, we write $xy=\{x,y\}$.

\begin{definition}[$k$-stable graph]\label{def:stablegraph}
A graph $(V,E)$ has the \emph{$k$-order property} if there exist $a_1,\ldots, a_k,b_1,\ldots, b_k\in V$ such that $a_ib_j\in E$ if and only if $i\leq j$. If $(V,E)$ does not have the $k$-order property, then it is \emph{$k$-stable}.
\end{definition}

Malliaris and Shelah \cite{Malliaris:2014go} proved that for any $\epsilon>0$, if a large finite graph is $k$-stable, then it has an $\epsilon$-regular partition with no irregular pairs, with all pairs having density close to $0$ or $1$, and with the number of parts polynomial in $1/\epsilon$.  Subsequent work showed that other model-theoretic dividing lines also correspond to improved regularity lemmas, see for instance \cite{Alon:2007gu,Lovasz:2010tm,Chernikov:2016we,Chernikov:2015tt,Pillay:oRBBGvRc,Basu:2009cr}\footnote{The work in \cite{Chernikov:2016we,Chernikov:2015tt,Basu:2009cr} generalizes prior work in discrete geometry, see for instance \cite{Alon:2005hj,Fox:2012us}}.  

The goal of this paper is to continue this line of investigation, this time in the setting of finite groups.  Specifically we will consider the implications of stability on the global structure of distinguished sets in finite-dimensional vector spaces over a fixed prime field. From now on, let $G$ be a finite abelian group.\footnote{The subsequent definitions and simple lemmas are also valid in infinite groups.} Given $A\subseteq G$, the \emph{Cayley graph of $A$ in $G$, $\Gamma(G,A)$} is the graph with vertex set $G$ and edge set $\{xy: x+y\in A\}$. When the group $G$ is clear from the context, we will write $\Gamma_A$ for $\Gamma(G,A)$. 

\begin{definition}[$k$-stable subset]\label{def:stableset}
Given $A\subseteq G$, $A$ is said to have the \emph{$k$-order property} if the graph $\Gamma(G,A)$ has the $k$-order property. We say that $A$ is \emph{$k$-stable} if the graph $\Gamma(G,A)$ is $k$-stable. 
\end{definition}

Given $A\subseteq G$, let $\neg A:=G\setminus A$. Note that if $A$ is $1$-stable, then $A=\emptyset$. We will therefore restrict our attention to the case $k\geq 2$ throughout. We now state a few further facts, some of which we shall need later on, and some of which we believe are useful for the reader unfamiliar with the notion of stability. Our first lemma states that translates of a stable set $A$ (and their complements) are stable. 

\begin{lemma}[Translates and complements of stable sets are stable]\label{lem:stillstable}
Suppose $A\subseteq G$ is $k$-stable and $g\in G$.  Then $A+g$ is $k$-stable and $\neg A+g$ is $(k+1)$-stable.
\end{lemma}
\begin{proof}
Suppose that $A+g$ is not $k$-stable.  Then there are elements $a_1,\ldots, a_k,b_1,\ldots, b_k$ such that $a_i+b_j\in A+g$ if and only if $i\leq j$.  This implies that $a_i-g+b_j\in A$ if and only if $i\leq j$.  Let $a'_1:=a_1-g,\ldots, a'_k:=a_k-g$ and $b'_1:=b_1,\ldots, b_k':=b_k$.  We have thus shown that $a'_i+b_j'\in A$ if and only if $i\leq j$, contradicting the assumption that $A$ is $k$-stable.  Hence $A+g$ is $k$-stable. 

Now suppose that $\neg A+g$ is not $(k+1)$-stable.  Then there are elements $a_1,\ldots, a_{k+1}$, $b_1,\ldots, b_{k+1}$ such that $a_i+b_j\in \neg A+g$ if and only if $i\leq j$.  Let $a_1':=a_2,\ldots, a_k':=a_{k+1}$ and $b_1':=b_1-g,\ldots, b_k':=b_k-g$.  Then $a_i'+b_j'\in A$ if and only if $j\leq i$, contradicting the assumption that $A$ is $k$-stable.
\end{proof}

It is not difficult to see that both intersections and unions of stable sets are stable. The following lemma is a quantitative statement to this effect. 

\begin{lemma}[Intersections and unions of stable sets are stable]\label{lem:stillstable1}
Let $h(k,\ell):=(k+\ell)2^{k+\ell}+1$.  If $A_0\subseteq G$ is $\ell$-stable and $A_1\subseteq G$ is $k$-stable, then $A_0\cap A_1$ is $h(k,\ell)$-stable. 
\end{lemma}

\begin{proof} Since $A_0\cap A_1=\neg (A_0\cup A_1)$, by Lemma \ref{lem:stillstable}, it suffices to show that $A_0\cup A_1$ is $h'(k,\ell)$-stable, where $h'(k,\ell):=(k+\ell)2^{k+\ell}$.  The proof of the latter fact is a standard Ramsey argument. We include it here for the sake of completeness, since we were unable to locate a suitable reference in the literature.

Suppose towards a contradiction that $A_0\cup A_1$ is not $h'(k,\ell)$-stable. Then there are $a_1,\ldots, a_{h'(k,\ell)},b_1,\ldots, b_{h'(k,\ell)}$ such that $a_i+b_j\in A_0\cup A_1$ if and only if $i\leq j$. Since $a_1+b_j\in A_0\cup A_1$ for all $1\leq j\leq h'(k,\ell)$, there is $i_1\in \{0,1\}$ and $D_1\subseteq \{b_j: a_1+b_j\in A_{i_1}\}$ such that $|D_1|\geq h'(k,\ell)/2$.  Let $D_1$ consist of $j_1<\ldots<j_{|D_1|}$, and define a new sequence $(a_1^1,\ldots, a^1_{|D_1|},b_1^1,\ldots, b_{|D_1|}^1):=(a_1,a_{j_2},\ldots, a_{j_{|D_1|}},b_{j_1},\ldots, b_{j_{|D_1|}})$.  By the pigeonhole principle, there is $i_2\in \{0,1\}$ and $D_2\subseteq \{b^1_j: j\geq 2, a^1_2+b^1_j\in A_{i_2}\}$ such that $|D_2|\geq |D_1|/2\geq h'(k,\ell)/4$.  Let $D_2$ consist of $s_1<\ldots<s_{|D_2|}$, and define a new sequence $(a_1^2,\ldots, a^2_{|D_2|},b_1^2,\ldots, b_{|D_2|}^2):=(a^1_1,a^1_2,a^1_{s_3},\ldots, a^1_{s_{|D_2|}},b^1_{s_1},\ldots, b^1_{s_{|D_1|}})$.   Continue this construction inductively. After $k+\ell$ steps, we will have constructed a new sequence $(a'_1,\ldots, a'_{t},b'_1,\ldots, b'_{t}):=(a^{k+\ell}_1,\ldots, a^{k+\ell}_{t},b^{k+\ell}_1,\ldots, b^{k+\ell}_{t})$ such that for each $1\leq j<s\leq t$, $a'_{s}+b'_j\notin A_0\cup A_1$ and for each $1\leq s \leq j\leq t$, $a'_{s}+b'_j\in A_{i_{s}}$, where $t\geq h'(k,\ell)/2^{k+\ell}=k+\ell$.  By the pigeonhole principle, either $|\{s: i_{s}=0\}|\geq \ell$ or $|\{s: i_{s}=1\}|\geq k$.  If $|\{s: i_{s}=0\}|\geq \ell$, delete all elements with indices not in $\{s: i_{s}=0\}$, and reindex the remaining elements, preserving their order, as $a^*_1,\ldots, a^*_{t^*},b^*_1,\ldots, b^*_{t^*}$.  Then $a^*_i+b^*_j\in A_{0}$ if and only if $i\leq j$, for each $1\leq i\leq j\leq t^*$, a contradiction since $t^*\geq \ell$ and $A_0$ is $\ell$-stable.  If, on the other hand, $|\{s: i_{s}=1\}|\geq k$, a similar argument yields a contradiction to the assumption that $A_1$ is $k$-stable. 
\end{proof}

At this point we owe the reader a first example of a stable set.

\begin{example}[Subgroups are 2-stable]\label{exa:subgroup}
Let $A\leqslant G$ be a subgroup of $G$. Assume that $a_1,a_2,b_1,b_2$ are such that $a_i+b_j\in A$ when $i\leq j$.  Then $(a_1+b_1)-(a_1+b_2)=b_1-b_2\in A-A=A$, and hence $a_2+b_1=(b_1-b_2)+(a_2+b_2)\in A+A=A$. This shows that $A$ is $2$-stable.
\end{example}

The converse is true under the additional assumption that $A\subseteq G$ contains $0$ and is symmetric: if such a set $A$ is 2-stable, then it must be a subgroup. To see this, fix any two elements $x,y\in A$.  Set $a_1:=-x$, $b_1:=x$, $a_2:=y$, and $b_2:=0$.  Then $a_1+b_1=0\in A$, $a_1+b_2=-x\in A$, and $a_2+b_2=y\in A$.  Since $A$ is $2$-stable, we must have $a_2+b_1=x+y\in A$.

In fact, it turns out that $k$-stable sets exhibit strong subgroup structure even for $k>2$, and a statement to this effect, Corollary \ref{cor:ff}, will be proved in Section \ref{sec:remarks}.

Taken in conjunction with Example \ref{exa:subgroup}, our next lemma shows that the bound in Lemma \ref{lem:stillstable1} can be improved when one of the sets is a subgroup. 

\begin{lemma}\label{lem:2stable}
If $A\subseteq G$ is $2$-stable and $B\subseteq G$ is $k$-stable, then $A\cap B$ is $k$-stable.
\end{lemma}

\begin{proof}
Suppose towards a contradiction that $a_1,\ldots, a_k, b_1,\ldots, b_k$ are elements in $G$ such that $a_i+b_j\in A\cap B$ if and only if $i\leq j$.  Fix $1\leq i<j\leq k$. Then all three of $a_i+b_i$, $a_i+b_j$, and $a_j+b_j$ must lie in $A$. Since $A$ is 2-stable, we must have that $a_j+b_i\in A$, which means that $a_j+b_i\notin B$.  It must therefore be the case that $a_i+b_j\in B$ if and only if $i\leq j$.  This contradicts the assumption that $B$ is $k$-stable.
\end{proof}

At this stage it is reasonable to enquire whether there are any non-trivial examples of $k$-stable sets with $k>2$. The following is an example of a set that has the $3$-order property but is $4$-stable. 

\begin{example}[$3$-order property but $4$-stable]
Let $n\geq 4$ and $A:=\{e_1,\ldots, e_n\}\subseteq \mathbb{F}_2^n$, where $e_i$ is the $i$th standard basis vector in $\mathbb{F}_2^n$. We first show that $A$ has the $3$-order property.  Let $b_1:=e_1, b_2:=e_2, b_3:=e_3$ and let $a_1:=0, a_2:=e_2+e_3, a_3:=e_3+e_4$.  We leave it to the reader to verify that for each $1\leq i,j\leq 3$, $a_i+b_j \in A$ if and only if $i\leq j$, and thus $A$ has the $3$-order property.

We now show that $A$ is $4$-stable.  Suppose towards a contradiction that there were distinct elements $a_1,\ldots, a_4,b_1,\ldots b_4$ such that $a_i+b_j\in A$ if and only if $i\leq j$. For each $i\leq j$, let $e_{ij}\in A$ be such that $a_i+b_j=e_{ij}$.  

Observe that by definition of $e_{ij}$, $b_2=a_1+e_{12}$, $b_3=a_1+e_{13}$, and $b_4=a_1+e_{14}$.  Since the elements $b_i$ are pairwise distinct, we must have that the elements $e_{12},e_{13},e_{14}$ are pairwise distinct.  Note further that
$$
\qquad a_1=b_2+e_{12}=b_3+e_{13}=b_4+e_{14}\qquad \hbox{ and  }\qquad a_2=b_2+e_{22}=b_3+e_{23}=b_4+e_{24}. \qquad
$$
Thus 
\begin{align*}
a_1+a_2&=b_2+e_{12}+b_2+e_{22}=e_{12}+e_{22},\\
a_1+a_2&=b_3+e_{13}+b_3+e_{23}=e_{13}+e_{23},\\
a_1+a_2&=b_4+e_{14}+b_4+e_{24}=e_{14}+e_{24}, 
\end{align*} 
from which it follows that $e_{12}+e_{22}=e_{13}+e_{23}=e_{14}+e_{24}$.  Since $e_{12}\neq e_{13}$, the first equality implies $e_{12}=e_{23}$ and $e_{22}=e_{13}$.  Since $e_{13}\neq e_{14}$, the second equality implies $e_{13}=e_{24}$ and $e_{23}=e_{14}$.  But now we have shown $e_{12}=e_{23}=e_{14}$, a contradiction.  Thus $A$ is $4$-stable.
\end{example}

Stable sets enjoy an interesting covering property. The following statement to this effect is a quantitative version of Lemma 5.1 in \cite{Poizat:2001vx}.

\begin{lemma}\label{lem:cover}
If $A\subseteq G$ is $k$-stable, then either $G$ is covered by $2k+1$ translates of $A$, or $G$ is covered by $2k+1$ translates of $\neg A$.
\end{lemma}
\begin{proof} 
Suppose towards a contradiction that $G$ is not covered by $2k+1$ translates of $A$, and that $G$ is not covered by $2k+1$ translates of $\neg A$.  We first inductively build a sequence $a_0,\ldots, a_{2k}$, $b_0,\ldots, b_{2k}$ such that $a_i+b_j\notin A$ if $i<j$ and $a_i+b_j\in A$ if $j<i$.

Choose $a_0=b_0$ to be any element of $G$. Assume now that $0\leq \ell\leq 2k-1$, and suppose that we have inductively chosen $a_0,\ldots, a_{\ell}$, $b_0,\ldots, b_{\ell}$ so that for all $0\leq i, j\leq \ell$, $a_i+b_j\notin A$ if  $i<j$ and $a_i+b_j\in A$ if  $j<i$.  Since $G$ is not covered by $2k+1$ translates of $A$ nor by $2k+1$ translates of $\neg A$, and since $\ell+1\leq 2k+1$,
\begin{align*}
G\neq (A-a_0) \cup \ldots \cup (A-a_{\ell})\text{ and } G\neq (\neg A-b_0) \cup \ldots \cup (\neg A-b_{\ell}).
\end{align*}
Choose $b_{\ell+1}\in G\setminus ((A-a_0) \cup \ldots \cup (A-a_{\ell}))$ and $a_{\ell+1}\in G\setminus ((\neg A-b_0) \cup \ldots \cup (\neg A-b_{\ell}))$.  Note that for each $0\leq i\leq \ell$, $a_i+b_{\ell+1}\notin A$ and $a_{\ell+1}+b_i\in A$.  Combining this with the inductive hypothesis, we have that for all $0\leq i, j\leq \ell+1$, $a_i+b_j\notin A$ if $i<j$ and $a_i+b_j\in A$ if $j<i$. 

After $2k+1$ steps, we have elements $a_0,\ldots, a_{2k}$, $b_0,\ldots, b_{2k}$ such that for all $0\leq i,j\leq 2k$, $a_i+b_j\notin A$ if $i<j$ and $a_i+b_j\in A$ if $j<i$. By the pigeonhole principle, 
\begin{align*}
\text{ either }|\{0\leq i\leq 2k: a_i+b_i\in A\}|\geq k \text{ or }|\{0\leq i\leq 2k: a_i+b_i\notin A\}|\geq k+1.
\end{align*}
If $|\{0\leq i\leq 2k: a_i+b_i\in A\}|\geq k $, choose $i_1<\ldots<i_{k}$ in $\{0\leq i\leq 2k: a_i+b_i\in A\}$ and set $a_1':=a_{i_1},\ldots, a_k':=a_{i_k}$ and $b_1':=b_{i_1},\ldots, b_k':=b_{i_{k}}$. Then for all $1\leq i,j\leq k$, $a_i'+b_j'\in A$ if and only if $j\leq i$, contradicting our assumption that $A$ is $k$-stable. On the other hand, if $|\{0\leq i\leq 2k: a_i+b_i\notin A\}|\geq k+1$, an analogous argument shows that $\neg A$ has the $(k+1)$-order property, which is a contradiction by Lemma \ref{lem:stillstable}.
\end{proof}

As an immediate corollary, we obtain that when $A$ is stable, then either $A$ or $\neg A$ has small doubling. For the reader familiar with the Freiman-Ruzsa theorem \cite{Ruzsa:1999uq}, this provides further evidence towards the thesis that being stable is closely connected to being close to a subgroup (see Section \ref{sec:remarks}).

\begin{corollary}\label{cor:cover}
If $G$ is finite and $A\subseteq G$ is $k$-stable, then one of the following holds.
\begin{enumerate}[(a)]
\item $|A+A|\leq (2k+1)|A|$ or 
\item $|\neg A+\neg A|\leq (2k+1)|\neg A|$.
\end{enumerate}
\end{corollary}
\begin{proof}
Lemma \ref{lem:cover} implies that either $|G|\leq (2k+1)|A|$ or $|G|\leq (2k+1)|\neg A|$.
\end{proof}

We observe that the converse to Lemma \ref{lem:cover} is false. This is illustrated by the following example of a group $G$ and a set $A\subseteq G$ with the property that three translates of $A$ cover $G$ but $A$ has the $(\log_2 |G|-1)$-order property.

\begin{example}\label{ex:convfalse}
Let $B\subseteq G=\mathbb{F}_2^n$ consist of $\{e_i+e_j: i\leq j\}$, where $\{e_1,\ldots, e_n\}$ are the standard basis vectors of $G$.  Clearly $B$ has the $n$-order property, and thus $A:=\neg B$ has the $(n-1)$-order property.  We claim that $G\subseteq A\cup (A+e_1)\cup (A+e_n)$.  Indeed, if $x\in G\setminus A$, then $x=e_i+e_j$ for some $i> j$. If $2\leq j< i$, we have $e_i+e_j=(e_1+e_i+e_j)+e_1 \in A+e_1$, and if $j=1<i$, we have $e_1+e_i=(e_1+e_i+e_n)+e_n\in A+e_n$. Thus $A$ has the $(n-1)$-order property, even though $G$ is covered by only three translates of $A$.
\end{example}

\section{Regularity and goodness}\label{sec:regularity}

Now that we have elucidated the notion of stability to the extent necessary for the remainder of this paper, we turn our attention to the concept of regularity. As sketched in the introduction, a regularity lemma allows one to split a mathematical object into a bounded number of clusters such that each cluster (or pair of clusters) behaves roughly like a random object. In the case of a graph $\Gamma=(V,E)$ with vertex set $V$ and edge set $E$, Szemer\'edi proved as part of his groundbreaking work on long arithmetic progressions in dense subsets of the integers \cite{Szemeredi:1975le} that given any $\epsilon>0$, one can partition $V$ into a bounded number of classes $V_1,\dots,V_k$ with $k\leq k_0(\epsilon)$ such that for all but an $\epsilon$-fraction of pairs $(i,j)\in[k]^2$ the pair $(V_i,V_j)$ is $\epsilon$-regular in the following sense. 

\begin{definition}[$\epsilon$-regular pair]\label{def:regularpair}
Let $U,W\subseteq V$, and let $\epsilon>0$. Then $(U,W)$ is said to be an \emph{$\epsilon$-regular pair} if for all subsets $U'\subseteq U$, $W'\subseteq W$ with $|U'|\geq \epsilon |U|$, $|W'|\geq \epsilon |W|$ , we have $|d(U',W')-d(U,W)|<\epsilon$.
\end{definition}

It was shown by Gowers that in general $k_0(\epsilon)$ must grow as a tower of height proportional to $\e^{-1}$, and by Malliaris and Shelah that it can be taken to be polynomial in $\e^{-1}$ if the graph is stable. Moreover, they showed that in the case of stable graphs, one can guarantee that all pairs in the partition are $\e$-regular, and furthermore that the density between each pair is either close to 0 or close to 1. We give a discussion of the proof technique at the start of Section \ref{sec:proof}.

We shall not use the statement of the regularity lemma in graphs (or stable graphs for that matter), but Definition \ref{def:regularpair} shall make an appearance later on in this section.  

Indeed, in this paper we are exclusively concerned with the notion of \emph{arithmetic} regularity, first introduced by Green in \cite{Green:2005kh}. 
Before giving a precise definition of what we mean by arithmetic regularity for the purposes of this paper, we need to set up some notation.

As is common in arithmetic combinatorics, given any subset $B\subseteq G$, let its \emph{characteristic measure} $\mu_B$ be defined as $\mu_B(x):=(|G|/|B|)1_B(x)$, where $1_B$ is the indicator function of $B$. Note that the normalisation is chosen so that $\E_{x \in G} \mu_B(x)=1$, where $\E_{x \in G}$ denotes the sum over all elements $x\in G$, normalised by the size $|G|$ of the group $G$. 

Further, given a subset $A\subseteq G$ and an element $y\in G$, write
\[f^y_{B,A}(x):=(1_{(A-y)\cap B}(x)-\alpha_{y+B}) \mu_B(x),\]
where $\alpha_{y+B}:=|(A-y)\cap B|/|B|$ is the density of $A$ on the translate $y+B$. Observe that $\E_{x \in G} f^y_{B,A}(x)=0$, i.e. we can think of $f^y_{B,A}$ as being the balanced indicator function of $A$ relative to $y+B$. 

We shall make some mild use of the Fourier transform on the group $G:=\F_p^n$ beyond the mere statement of the arithmetic regularity lemma. Let $\widehat{G}$ denote the group of characters on $G$, which take the form $\gamma: G\rightarrow \T$, $\gamma(x)=\omega^{x\cdot t}$ for some $t\in \F_p^n$. Here $\omega:=\exp(2\pi i/p)$ is a $p$th roof of unity, and $\cdot$ denotes the usual scalar product. It is not difficult to see that $\widehat{G}$ is in fact isomorphic to $G$ itself.

For each $t\in\widehat{G}=\F_p^n$, let the Fourier transform of $f:\F_p^n \rightarrow \C$ at $t$ be defined by
\[\widehat{f}(t):=\E_{x\in G} f(x)\omega^{x\cdot t}.\]
Orthogonality of the characters leads straightforwardly to the inversion formula
\[f(x)=\sum_{t\in \widehat{G}} \widehat{f}(t)\omega^{-x\cdot t}.\]
It is also easy to check that Parseval's identity holds, that is,
\begin{equation}\label{eq:parseval}
\E_{x\in G} |f(x)|^2=\sum_{t\in \widehat{G}} |\widehat{f}(t)|^2.
\end{equation}
Finally, observe that since $\E_{x \in G} f^y_{H,A}(x)=0$, we have that $\widehat{f_{H,A}^y}(t)=0$ whenever $t\in H^\perp$.

Here then is the Fourier-analytic notion of regularity we shall use.

\begin{definition}[$\epsilon$-uniform with respect to $B$]\label{def:regular}
Let $A, B\subseteq G$, and let $y\in G$.  We say that $y$ is \emph{$\epsilon$-uniform for $A$ with respect to $B$} if $\sup_{t\in \widehat{G}}|\widehat{f^y_{B,A}}(t)|\leq \epsilon$. 
\end{definition}

We shall apply this concept exclusively in the case where $B=H$ is a subspace of $G=\F_p^n$, which means that we only care about the size of the Fourier coefficients of $\widehat{f^y_{H,A}}(t)$ when $t\notin H^\perp$. When $A$ is clear from the context, we shall simply refer to $y$ as \emph{$\epsilon$-uniform with respect to $B$} and omit the subscript $A$. Note that in \cite{Green:2005kh}, an element $y$ satisfying the condition in Definition \ref{def:regular} was said to be an ``$\epsilon$-regular value with respect to $A$". We deliberately rename the concept here so as to avoid any confusion with the notion of a regular Bohr set in forthcoming work.
 
\begin{definition}[totally $\epsilon$-uniform]\label{def:totregular}
Let $A\subseteq G:=\F_p^n$ and $H\leqslant G$ be a subspace. We say that $H$ is \emph{totally $\epsilon$-uniform for $A$} if every $y\in G$ is $\epsilon$-uniform for $A$ with respect to $H$.  In other words, $H$ is totally $\epsilon$-uniform for $A$ if $|\widehat{f^y_{H,A}}(t)|\leq \epsilon$ for all $y\in G$ and $t\notin H^\perp$.
\end{definition}

Observe that $H$ being totally $\epsilon$-uniform for $A$ in the sense of Definition \ref{def:totregular} is stronger than $H$ being ``$\epsilon$-regular for $A$" according to \cite{Green:2005kh}, where a small number of non-uniform values $y\in G$ were permitted.

The goal of this paper is to show that when $A$ is $k$-stable for some $k\geq 2$, then we can find a totally $\epsilon$-uniform subspace $H$ whose parameters depend only on $k$ and $\epsilon$. In fact, we shall establish the existence of a subspace $H$ with an even stronger property, which we characterise as follows.

\begin{definition}[$\epsilon$-good]\label{def:good}
Let $A, B\subseteq G$, and let $y\in G$. We say that $y$ is \emph{$\epsilon$-good for $A$ with respect to $B$} if $|(A-y)\cap B|\leq \epsilon |B|$ or $|B\setminus (A-y)|\leq \epsilon |B|$. We say that $B$ is \emph{$\epsilon$-good for $A$}  if $y$ is $\epsilon$-good for $A$ with respect to $B$ for all $y\in G$.\footnote{For those familiar with \cite{Malliaris:2014go}, this is equivalent to saying that $B$ is an $\epsilon$-good subset in the Cayley graph $\Gamma(G,A)$.}
\end{definition}

Again, we shall use this definition only in the case where $B=H$ is a subspace of $G$, and often drop the reference to $A$ when this causes no ambiguity. Of course, the notions in Definitions \ref{def:totregular} and \ref{def:good} are intimately related. In particular, an $\epsilon$-good subspace will also be totally $\epsilon'$-uniform for some $\epsilon'$ which goes to $0$ with $\epsilon$.

\begin{lemma}[Good implies totally uniform]\label{lem:exctototregtoy}
Let $A\subseteq G:=\F_p^n$ and let $H\leqslant G$ be a subspace. If $H$ is $\epsilon$-good for $A$, then it is totally $\epsilon (p+1)$-uniform for $A$.
\end{lemma}

\begin{proof}
Fix any $y\in \F_p^n$. We want to show that for any $t\notin H^\perp$,
\[|\widehat{f_H^y}(t)|=\left|\E_x (1_{(A-y)\cap H}(x)-\alpha_{y+H})\mu_H(x)\omega^{t\cdot x}\right|=\left|\E_{x\in H} (1_{(A-y)\cap H}(x)-\alpha_{y+H})\omega^{t\cdot x}\right|\]
is at most $\e (p+1)$, provided that $H$ is $\e$-good for $A$. For fixed $t\notin H^\perp$ define $H':=H\cap\langle t\rangle^\perp$, and partition $H$ into cosets $H_j$ of $H'$, $j=0,1,\dots,p-1$ such that $x\cdot t=j$ for all $x\in H_j$. Let $\alpha_j:=|(A-y)\cap H_j|/|H_j|$ and note that $|H_j|=|H|/p$. It follows that 
\[|\widehat{f_H^y}(t)|=\left|\E_j \omega^j \E_{x\in H_j} (1_{(A-y)\cap H}(x)-\alpha_{y+H})\right|=\left|\E_j \omega^j (\a_j-\a_{y+H})\right|.\]
Because $H$ is $\e$-good for $A$, either $|(A-y)\cap H|\leq \e |H|$ or $|H\setminus(A-y)|\leq \e|H|$. Suppose first that $|(A-y)\cap H|\leq \e |H|$ holds, that is, $\a_{y+H}\leq \e$. Then by the triangle inequality we have
\[|\widehat{f_H^y}(t)|=\left|\E_j \omega^j (\a_j-\a_{y+H})\right|\leq \E_j \a_j+\a_{y+H}=2\a_{y+H}\leq 2\epsilon,\]
so since $p\geq 2$ we are done. Suppose now that $|H\setminus(A-y)|\leq \e|H|$, that is, $\a_{y+H}\geq 1- \e$. Since $\E_j\a_j=\a_{y+H}$ and $\a_j\leq 1$,
\[\a_j=\sum_{\ell=0}^{p-1}\a_\ell-\sum_{\ell\neq j}\a_\ell\geq p\cdot\a_{y+H}-(p-1)\geq p(1-\e)-(p-1)=1-\e p,\]
and thus $1-\a_j\leq \e p$. It follows that 
\[|\widehat{f_H^y}(t)|\leq \E_j \left| (\a_j-\a_{y+H})\right|=\E_j \left| (1-\a_j)-(1-\a_{y+H})\right|\leq \e p+\e=(p+1)\e,\]
which concludes the proof.
\end{proof}

It will be convenient to express our first auxiliary result in the language of graphs. The reader should bear in mind that we will be applying it to the Cayley-type graph $\Gamma_A=\Gamma(G,A)$, and that it could therefore be rephrased without any reference to the graph setting.  We begin with a lemma which says that regular pairs in stable graphs must have density close to $0$ or $1$.  This is a direct corollary of the induced embedding lemma (a special case of Theorem 14 in \cite{Komlos:2002gf}), and indeed the bounds stated in Lemma \ref{lem:trivregpairs} below are those that arise from this approach. For the benefit of the reader unfamiliar with the embedding lemma we have included a direct proof here, which yields ever so slightly weaker bounds. Given a graph $\Gamma=(V,E)$ and $x\in V$, let $N(x):=\{y\in V: xy\in E\}$.

\begin{lemma}[Regular pairs in stable graphs have density close to 0 or 1]\label{lem:trivregpairs}
For all integers $k\geq 2$ and all $0<\epsilon<\epboundtrivregpairs$, there exists $m=m(k,\epsilon)$ such that the following holds.  Suppose that $\Gamma=(V,E)$ is a $k$-stable graph and that $X, Y\subseteq V$ are subsets such that $|X|=|Y|\geq m$ and $(X,Y)$ is an $\epsilon$-regular pair.  Then either $d(X,Y)\leq \eptrivregpairs{\epsilon}{k}$ or $d(X,Y)\geq 1-\eptrivregpairs{\epsilon}{k}$.
\end{lemma}
\begin{proof} We shall prove the immaterially weaker statement that for every integer $t\geq 1$ and $0<\epsilon<(1/2)^{2t+2}$, there exists $m=m(t,\epsilon)$ such that the following holds.  Suppose that $\Gamma=(V,E)$ is a $t$-stable graph and that $X, Y\subseteq V$ are subsets such that $|X|=|Y|\geq m$ and $(X,Y)$ is an $\epsilon$-regular pair.  Then either $d(X,Y)\leq \epsilon^{1/(2t+2)}$ or $d(X,Y)\geq 1-\epsilon^{1/(2t+2)}$.

Let $k:=2t+2$, and suppose that $m$ is sufficiently large compared to $k$ and $1/\epsilon$ (to be determined later). Note that since $\epsilon<(1/2)^k$, $\epsilon^{1/k}-\epsilon>\epsilon^{2/k}$. 

Suppose towards a contradiction that $X, Y\subseteq V$ are subsets such that $|X|=|Y|\geq m$, the pair $(X,Y)$ is $\epsilon$-regular, and $\epsilon^{1/k} < d(X,Y)< 1-\epsilon^{1/k}$.  We build by induction a sequence $x_1,\ldots, x_t,y_1,\ldots, y_t$ such that $x_iy_j\in E$ if and only if $i\leq j$, contradicting the assumption that $\Gamma$ is $t$-stable.

To start, observe that since $d(X,Y)>\epsilon^{1/k}$, there is $x_1\in X$ such that $|N(x)\cap Y|\geq \epsilon^{1/k} |Y|$.  Let $Y_1:=N(x)\cap Y$ and let $Z:=X\setminus \{x_1\}$.  Because $m$ is large, because $\epsilon<\epsilon^{1/k}$, and by definition of $Z$ and $Y_1$, we have that $|Y_1|\geq \epsilon|Y|$ and $|Z|=|X|-1\geq \epsilon |X|$.  Thus, since $(X,Y)$ is $\epsilon$-regular, 
$$
d(Z,Y_1)\leq d(X,Y)+\epsilon < 1-\epsilon^{1/k}+\epsilon.
$$
It follows that there is $y_1\in Y_1$ such that $|Z\setminus N(y_1)|\geq (\epsilon^{1/k}-\epsilon)|Z|$.  Let $X_1:=Z\setminus N(y_1)$.  Note that $|Y_1|\geq \epsilon^{1/k} |Y|$ and $|X_1|\geq (\epsilon^{1/k}-\epsilon)(|X|-1)\geq \epsilon^{2/k}|X|$, since $|X|\geq m$ is sufficiently large. 

Now assume that $1\leq i<t$ and suppose that we have inductively constructed $x_1,\ldots, x_i$, $y_1,\ldots, y_{i}$ and sets $X_i\subseteq X$, $Y_i\subseteq Y$ such that the following hold.
\begin{enumerate}[(i)]
\item  $|X_i|\geq \epsilon^{2i/k}|X|$ and $|Y_i|\geq \epsilon^{(2i-1)/k}|Y|$;
\item $x_i\in X_i$, $y_i\in Y_i$;
\item  for each $1\leq j, s\leq i$, $x_jy_s\in E$ if and only if $j\leq s$;
\item $Y_i\subseteq N(x_1)\cap \ldots \cap N(x_i)\cap Y$, $X_i\subseteq X\setminus (N(y_1)\cup \ldots \cup N(y_{i}))$.
\end{enumerate}
By (i) and since $(X,Y)$ is $\epsilon$-regular, we have that 
$$
d(X_i,Y_i)\geq d(X,Y)-\epsilon > \epsilon^{1/k} -\epsilon.
$$
Thus there is $x_{i+1}\in X_i$ such that 
$$
|N(x_{i+1})\cap Y_i|\geq (\epsilon^{1/k}-\epsilon)|Y_i|\geq (\epsilon^{1/k}-\epsilon)\epsilon^{(2i-1)/k}|Y|>(\epsilon^{2/k})\epsilon^{(2i-1)/k}|Y|=\epsilon^{(2(i+1)-1)/k}|Y|\geq \epsilon |Y|.
$$
Let $Y_{i+1}:=N(x_{i+1})\cap Y_i$ and let $Z:=X_i\setminus \{x_{i+1}\}$.  Note that as before, $|Z|=|X_{i}|-1\geq \epsilon^{2i/k}|X|-1\geq \epsilon |X|$ because $m$ is large.
Thus, $|Z|\geq \epsilon |X|$ and $|Y_{i+1}|\geq \epsilon |Y|$, so by $\epsilon$-regularity of $(X,Y)$, we have
$$
d(Z,Y_{i+1})\leq d(X,Y)+\epsilon < 1-\epsilon^{1/k}+\epsilon.
$$
Thus there is $y_{i+1}\in Y_{i+1}$ such that $|Z\setminus N(y_{i+1})|\geq (\epsilon^{1/k}-\epsilon)|Z|$.  Set $X_{i+1}:=Z\setminus N(y_{i+1})$.  By definition of $X_{i+1}$ and our induction hypothesis we have
\begin{align*}
|X_{i+1}|\geq (\epsilon^{1/k}-\epsilon)|Z|\geq (\epsilon^{1/k}-\epsilon)(|X_i|-1)\geq (\epsilon^{1/k}-\epsilon)(\epsilon^{2i/k}|X|-1)&\geq \epsilon^{2/k}(\epsilon^{2i/k}|X|)\geq \epsilon m,
\end{align*}
where the last two inequalities are because $m$ is large, and $2(i+1)/k\leq 1$. This finishes the inductive step of our construction. After $t$ steps we have constructed $x_1,\ldots, x_t,y_1,\ldots, y_t$ such that $x_iy_j\in E$ if and only if $i\leq j$, a contradiction.

Examining the largeness assumptions placed on $m$ in more detail, we find that it suffices to have 
\begin{enumerate} [(a)]
\item $m>k/2-2$;
\item $m$ sufficiently large such that for all $N\geq \epsilon m$, we have $(\epsilon^{1/k}-\epsilon)N-1\geq \epsilon^{2/k}N$;
\item $m$ sufficiently large such that for all $N\geq \epsilon^{2(t-1)/k}m$, we have $N-1\geq \epsilon m$.
\end{enumerate}
The lemma then follows as stated at the start of the proof.
\end{proof}

We now relate the notion of $\e$-uniformity for $A$ to the existence of certain regular pairs in $\Gamma_A$. The following lemma to this effect is straightforward, and essentially contained in Section 9 of \cite{Green:2005kh}. We include it here for completeness and because our definitions differ slightly from those in that paper.

\begin{lemma}[Subspaces induce regular pairs]\label{lem:indregtoy}
Let $\epsilon>0$ and $A\subseteq G:= \mathbb{F}_p^n$. Suppose that $H\leq G$ is a subspace, and that $y\in G$ is $\epsilon$-uniform for $A$ with respect to $H$.  Then $(H,H+y)$ is a $\epindregtoy$-regular pair in $\Gamma_A$, and $d(H,H+y)=|(A-y)\cap H|/|H|$.
\end{lemma}

\begin{proof} The latter statement follows straight from the definitions. To see that $(H,H+y)$ is a $\e^{1/2}$-regular pair in $\Gamma_A$, let $U,V\subseteq H$ be two subsets of $H$ of density at least $\e^{1/2}$, and consider the edge density $d(U, V+y)$ between $U$ and $V+y$, which may be estimated using the Fourier transform. Indeed, we write 
\begin{equation}\label{eq:regpair}
d(U,V+y)-\a_{y+H}=\frac{|H||G|}{|U||V|}\E_{u,v} f_H^y(u+v)1_U(u)1_V(v),
\end{equation}
where the expectation in $u$ and $v$ is in absolute value equal to
\[|\sum_{s \notin H^\perp} \widehat{f_H^y}(s)\widehat{1_U}(s)\widehat{1_V}(s)|\leq \sup_{s\notin H^\perp} |\widehat{f_H^y}(s)|\left(\sum_s|\widehat{1_U}(s)|^2\right)^{1/2}\left(\sum_s|\widehat{1_V}(s)|^2\right)^{1/2}.\]
Since $y$ is $\e$-uniform for $A$ with respect to $H$, this expression is, by Parseval's identity (\ref{eq:parseval}), bounded above by
\[\e|U|^{1/2}|V|^{1/2}/|G|.\]
It follows that
\[\left|d(U,V+y)-\a_{y+H}\right|\leq \e\frac{|H||G|}{|U||V|}\frac{|U|^{1/2}|V|^{1/2}}{|G|}=\frac{\e |H|}{|U|^{1/2}|V|^{1/2}}\leq \e^{1/2},\]
where the latter inequality follows from the fact that both subsets $U$ and $V$ have density at least $\e^{1/2}$ in $H$.
\end{proof}

We now prove the main result of this section, which can be viewed as a partial converse to Lemma \ref{lem:exctototregtoy} under the condition that $A$ is stable. 

\begin{proposition}[Uniform implies good for stable sets]\label{lem:goodlemtoy}
Let $k\geq 2$, $0<\epsilon<\boundgoodlemtoy{k}$, $M\geq 0$, and let $p$ be a prime.  Then there is $N=N(\epsilon, k,M,p)$ such that for all $n\geq N$ the following holds. Suppose that $A\subseteq \F_p^n$ is $k$-stable and that $H\leq \F_p^n$ is a subspace of codimension at most $M$. If $y\in \F_p^n$ is $\epsilon$-uniform for $A$ with respect to $H$, then $y$ is $\epgoodlemtoy{\epsilon}{k}$-good for $A$ with respect to $H$.
\end{proposition}

\begin{proof}
Choose $N$ sufficiently large so that $n\geq N$ implies $p^{n-M}\geq m(k,\epindregtoy)$, where $m(k,\epindregtoy)$ is as in Lemma \ref{lem:trivregpairs}. Suppose now that $n\geq N$, $A\subseteq \F_p^n$ is $k$-stable, $H\leq \F_p^n$ has codimension at most $M$, and $y$ is $\e$-uniform for $A$ with respect to $H$. By Lemma \ref{lem:indregtoy}, $(H,H+y)$ is $\epindregtoy$-regular and $d(H,H+y)=|(A-y)\cap H|/|H|$.  Because $A$ is $k$-stable, $(H,H+y)$ is $\epindregtoy$-regular, and $|H|\geq m(k,\epindregtoy)$, Lemma \ref{lem:trivregpairs} implies that either $d(H,H+y)\leq \epgoodlemtoy{\epsilon}{k}$ or $d(H,H+y)\geq 1-\epgoodlemtoy{\epsilon}{k}$.  Thus $y$ is $\epgoodlemtoy{\epsilon}{k}$-good for $H$.
\end{proof}

\section{Building a tree in the absence of efficient regularity}\label{sec:proof}

In this section we prove our main result, Theorem \ref{thm:mainthmff} below, which states that if $A$ is $k$-stable, then there is an $\epsilon$-good subspace of codimension polynomial in $1/\epsilon$.  The proof proceeds by contradiction: we shall assume that there is no such $\epsilon$-good subspace, and use this assumption to build a model-theoretic configuration called a tree (see Definition \ref{def:treebound}), which in turn implies a large instance of the order-property.  This general strategy is based on the proof of the stable regularity lemma for graphs in \cite{Malliaris:2014go}. 

There are two key differences between our argument and that in \cite{Malliaris:2014go}.  First, we require an additional ingredient in the form of  Proposition \ref{prop:keylemtoy} below, which shows that a dense set which is $k$-stable has to contain most of a translate of a large subspace. This step is unnecessary in the case of graphs because there are no preferred substructures. Second, to prove Theorem 5.18 in  \cite{Malliaris:2014go} a version of the tree argument we use to prove Theorem \ref{thm:mainff} must be iterated several times, and each iteration generates a new part of the desired partition. We are able to conclude the proof having applied this tree argument only once.  This is because after one iteration we obtain the desired subgroup, the cosets of which automatically generate the other parts in the partition. These differences between the two stable regularity lemmas are rooted in the fact that they are not directly comparable. We refer the reader to the end of Section \ref{sec:remarks} for a more in-depth discussion of their relationship.

Here then is the statement of the aforementioned additional arithmetic ingredient.

\begin{proposition}[Stable sets are dense on subspaces]\label{prop:keylemtoy}
Let $k\geq 2$, $0<\epsilon<\boundgoodlemtoy{k}$, $M\geq 0$, and let $p$ be a prime.  Then there is $\nu=\nu(k,\epsilon, M,p)$ such that for all $n\geq \nu$ the following holds.  Suppose that $H\leq G:= \mathbb{F}_p^n$ is a subspace of codimension at most $M$, and that $A\subseteq G$ is $k$-stable with $|A\cap H|>\epgoodlemtoy{\epsilon}{k}|H|$.  Then there is a subspace $H'\leq H$ and $x\in G$ such that $H'$ has codimension at most $\lfloor2/\epsilon \rfloor$ in $H$ and $|A\cap (H'+x)|\geq (1-\epgoodlemtoy{\epsilon}{k})|H'|$.
\end{proposition}

\begin{proof}
Choose $\nu=N(\epsilon, k, M+\lfloor 2/\epsilon\rfloor,p)$, where $N(\epsilon, k, M+\lfloor 2/\epsilon\rfloor,p)$ is as in Proposition \ref{lem:goodlemtoy}. Assume that $n\geq \nu$, $A\subseteq G=\F_p^n$ is $k$-stable, $H\leq G$ has codimension at most $M$, and $\alpha:=|A\cap H|/|H|>\epgoodlemtoy{\epsilon}{k}$. We first show there is a subspace $H'\leq H$ of codimension at most $\lfloor 2/\epsilon\rfloor$ in $H$ together with an element $x\in G$ such that $x$ is $\epsilon$-uniform for $A$ with respect to $H'$ and $|(A-x)\cap H'|\geq \alpha |H'|$.  

Fix any $y_0\in G$, say $y_0=0$.  By a standard Fourier argument (see for example Theorem 1.1 in \cite{Green:2015wp}), if $y_0\in G$ is not $\e$-uniform for $A$ with respect to $H$, then there exist $x_0\in G$ and a subspace $H_0\leqslant H$ of codimension 1 in $H$ such that
\[\frac{|(A-y_0)\cap(H_0+x_0)|}{|H_0|}\geq \a_{y_0+H}+\e/2.\]
Set $y_1:=y_0+x_0$, so that $\alpha_{y_1+H_0}\geq \a_{y_0+H}+\e/2$. If $y_1$ is not $\epsilon$-uniform for $A$ with respect to $H_0$, the same Fourier argument implies there is $x_1\in G$ and a subspace $H_1\leqslant H_0$ of codimension 1 in $H_0$ such that
\[\frac{|(A-y_1)\cap(H_1+x_1)|}{|H_1|}\geq \a_{y_1+H_0}+\e/2\geq \alpha_{y_0+H}+\e.\]
Set $y_2:=y_1+x_1$ and note that the inequality above implies $\alpha_{y_2+H_1}\geq \alpha_{y_0+H}+\e$. Iterating this procedure and using the fact that the relative density of $A$ on a subspace cannot exceed 1, we find that there exists a subspace $H'\leq H$ of codimension at most $\lfloor 2/\epsilon\rfloor$ in $H$ and $x\in G$ such that $x$ is $\epsilon$-uniform with respect to $H'$ and 
\[|A\cap (H'+x)|=|(A-x)\cap H'|\geq \alpha|H'|>\epgoodlemtoy{\epsilon}{k}|H'|.\]
 Note $H'$ has codimension at most $M+\lfloor 2/\epsilon\rfloor$ in $G$.  Consequently, by Proposition \ref{lem:goodlemtoy}, because $A$ is $k$-stable, $x$ is $\epsilon$-uniform, and $\epsilon<\boundgoodlemtoy{k}$, we must have either 
\[|(A-x)\cap H'|\leq \epgoodlemtoy{\epsilon}{k}|H'|\text{ or }|H'\setminus (A-x)|\leq \epgoodlemtoy{\epsilon}{k}|H'|.\]
Since $|(A-x)\cap H'| >\epgoodlemtoy{\epsilon}{k}|H'|$, it must be the case that $|H'\setminus (A-x)|\leq \epgoodlemtoy{\epsilon}{k}|H'|$, and consequently, $|(A-x)\cap H'|\geq (1-\epgoodlemtoy{\epsilon}{k})|H'|$.  Since $|(A-x)\cap H'|=|A\cap (H'+x)|$, we are done.
\end{proof}

As mentioned above, our construction of the order property shall proceed somewhat indirectly, as in \cite{Malliaris:2014go}. Indeed, we shall be constructing a large ``tree" inside the graph $\Gamma_A$, rather than a large instance of the order property.\footnote{This is not a tree in the graph-theoretic sense, see Definition \ref{def:treebound} below.} In order to do so we shall need some more notation. Let $\Gamma=(V,E)$ be a graph.  Given $x\in V$ and $A\subseteq V$, set $N(x):=\{y: xy\in E\}$ and 
 \begin{align*}
N^1(x)&:=N(x), \text{ }N^0(x):=V\setminus N(x), \text{ }A^1:=A\text{ and }A^0:=V\setminus A.
\end{align*}
Observe that in the case where $\Gamma=\Gamma(G,A)$ for some finite abelian group $G$ and $A\subseteq G$, we have that for each $i\in \{0,1\}$ and $x,y\in G$, $N^i(x):=A^i-x$, and $y+N^i(x)=N^i(x-y)$.  Given an integer $n\geq 1$, define $2^n:=\{0,1\}^n$ and
\begin{align*}
2^{<n}:=\bigcup_{i=0}^{n-1} \{0,1\}^i,
\end{align*}
where $\{0,1\}^0:=\langle$ $\rangle$ is the empty string. Given $\eta, \eta'\in 2^{<n}$, we say that $\eta \trianglelefteq \eta'$ if and only if $\eta=\langle$ $\rangle$ or $\eta$ is an initial segment of $\eta'$.  We will write $\eta \triangleleft \eta'$ to mean that $\eta \trianglelefteq \eta'$ and $\eta \neq \eta'$.  Given $\eta \in \{0,1\}^i$, let $|\eta|=i$ denote the length of $\eta$ (the length of the empty string $\langle$ $\rangle$ is $0$). Given $\eta \in 2^n$, and $i\in \{0,1\}$, $\eta\wedge i$ denotes the element of $2^{n+1}$ obtained by adding $i$ to the end of $\eta$.  If $\eta=(\eta_1,\ldots, \eta_n) \in 2^n$ and $1\leq i\leq n$, let $\eta |_i:=(\eta_1,\ldots, \eta_i)$, and $\eta(i):=\eta_i$.  By convention, $\eta|_0:=\langle \rangle$ and $2^{<0}:=\emptyset$.  

\begin{definition}[Tree bound]\label{def:treebound}
Given a graph $\Gamma=(V,E)$, the \emph{tree bound} for $\Gamma$, denoted by $d(\Gamma)$, is the least integer $d$ such that there do not exist sequences $\langle a_{\eta}: \eta \in 2^d\rangle$, $\langle b_{\rho}: \rho \in 2^{<d}\rangle$ of elements of $V$ with the property that for each $\eta\in 2^d$ and $\rho\in 2^{<d}$, if $\rho \triangleleft \eta$, then $a_{\eta}b_{\rho}\in E$ if and only if $\rho\wedge 1\trianglelefteq \eta$.
\end{definition}

The following result, which appears in \cite{Wilfrid:1993wx}, relates the degree of stability of the graph to its tree bound. 

\begin{theorem}[Stability bounds height of trees]\label{thm:treefact}
For each integer $k$ there exists $d=d(k)<2^{k+2}-2$ such that if $\Gamma$ is a $k$-stable graph, then its tree bound $d(\Gamma)$ is at most $d$.  
\end{theorem}
Recall that by Lemma \ref{lem:stillstable1}, if $A$ is $k$-stable and $B$ is $y$-stable for some $k,y\geq 2$, then $A\cap B$ is $h(k,y)$-stable, where $h(x,y):=(x+y)2^{x+y}+1$. Furthermore, if $y=2$, then by Lemma \ref{lem:2stable} $A\cap B$ is $k$-stable. 

Given $k\geq 2$, let $f_k(y)$ be the function in one variable defined by $f_k(y):=h(k+1,y)$. For $i\geq 1$, let $f_k^i(y)$ denote the function obtained by applying $f_k$ $i$ times. That is, for any integer $i\geq -1$, $f_k(f_k^i(y))=f_k^{i+1}(y)$, where by convention we let $f_k^{-1}$ and $f_k^0$ denote the constant functions $f_k^{-1}(y):=2$ and $f^0_k(y):=k$, respectively. It follows, and shall be important later, that for any $i\geq -1$, if $A$ is $k$-stable and $B$ is $f_k^i(k)$-stable, then $A\cap B$ and $(\neg A)\cap B$ are both $f_k^{i+1}(k)$-stable.
 
Equipped with these preliminary remarks, we are now able to state the technical version of our main theorem.

\begin{theorem}[Main theorem]\label{thm:mainthmff}
For all $\mu\in (0,1)$, $k\geq 2$, and every prime $p$, there is $n_0=n_0(k,\mu,p)$ such that the following holds. Let $d=d(k)$ be as in Theorem \ref{thm:treefact}, and set $D:=f_k^{d}(k)$ where $f_k(y):=h(k,y)$ is as in Lemma \ref{lem:stillstable1}. Let $G:=\F_p^n$ with $n\geq n_0$, and suppose that $A\subseteq G$ is $k$-stable.  Then there exists a subspace of $H\leqslant G$ which is $\mu$-good for $A$ and which has codimension at most $\max\{d\lfloor 2(4/\mu)^{2D}\rfloor, d\lfloor 2(4)^{D^2}\rfloor\}$.
\end{theorem}

\begin{proof} We begin with some reductions and observations. Set $\epsilon:=\min\{(\mu/4)^{2D}, \epboundfact\}$, $m:=\lfloor 2/\epsilon \rfloor$, and $n_0:=\max\{\nu(f^t(k),\epsilon, mt,p): 0\leq t\leq d\}$, where $\nu(f^t(k),\epsilon, mt,p)$ is given by Proposition \ref{prop:keylemtoy} and we drop the subscript on $f$ for clarity.  Assume that $n\geq n_0$, and that $A\subseteq G:=\F_p^n$ is $k$-stable.  Observe that since $\epsilon \leq (\mu/4)^{2D}$, any $4\epsilon^{1/2D}$-good subspace is also $\mu$-good.  

Consequently, it suffices to find a subspace $H$ which is $\epmainthmff{\epsilon}{D}$-good for $A$ and which has codimension at most $dm$. We shall use throughout that if $0\leq t < d$, then $f^t(k)<D$. We shall also use the fact that this, along with our assumption on $\epsilon$, implies that $\epsilon<(1/f^t(k))^6$ for all $0\leq t\leq d$. Finally, we invite the reader to verify the following fact.
\begin{align}\label{fact000}
\text{If $D>v,u\geq 2$ are integers and $0<\epsilon<\epboundfact$, then $\epsilon^{1/2D}-\epsilon^{1/2u}>\epsilon^{1/2v}$.}
\end{align}

Suppose towards a contradiction that there is no subspace of codimension at most $dm$ which is $\epmainthmff{\epsilon}{D}$-good for $A$.  We will now simultaneously construct four sequences, each indexed by $2^{\leq d}$, as follows.
\begin{enumerate}[(a)]
\item $\langle H_{\eta}: \eta\in 2^{ \leq d}\rangle$, where each $H_{\eta}\leq G$;
\item $\langle g_{\eta}: \eta\in 2^{\leq d}\rangle$ and $\langle x_{\eta}: \eta\in 2^{\leq d} \rangle$, where each $g_{\eta}$ and $x_{\eta}$ are elements of $G$; 
\item $\langle X_{\eta}: \eta\in 2^{\leq d}\rangle$, where each $X_{\eta}\subseteq G$.
\end{enumerate}
These sequences will satisfy the following for each $0\leq t\leq d$, and all $\eta \in 2^t$:
\begin{enumerate}[(i)]
\item $H_{\eta}\leq G$ has codimension at most $mt$; \label{parameterbounds}
\item \label{ghypp} for each $i\in \{0,1\}$, $|N^i(g_{\eta})\cap H_{\eta}|\geq \epmainthmff{\epsilon}{D}|H_{\eta}|$;
\item \label{largeintersection} $|X_{\eta}\cap H_{\eta}|\geq (1-\epmainthmff{\epsilon}{f^{t-1}(k)})|H_{\eta}|$;
\item $X_{\eta}$ is $f^{t-1}(k)$-stable in $G$;\label{stabilityhypothesis}
\item if $\eta=\sigma \wedge i$, $X_{\eta}=N^i(g_{\sigma}+x_{\sigma\wedge i})\cap (X_{\sigma}-x_{\sigma \wedge i})$;\label{defofset}
\item for all $s<t$, and $\sigma \in 2^{s}$ satisfying $\sigma \triangleleft \eta$, the following holds for all $x\in X_{\eta}$:\label{importanthypothesis}
\begin{align}\label{5}
x+g_{\sigma}+ x_{\eta|_{s+1}}+\ldots+x_{\eta|_{t}} \in A \Leftrightarrow \eta(s+1)=1.
\end{align}
\end{enumerate}
We proceed by induction on $t$. For the base case $t=0$, set $x_{<>}:=0$, $g_{<>}:=0$, $H_{<>}:=G$, and $X_{<>}:=G$.  Note that $X_{<>}$ is $2=f^{-1}(k)$-stable. Since $G$ is not $\epmainthmff{\epsilon}{D}$-good for $A$, for each $i\in \{0,1\}$, $|N^i(g_{<>})\cap X_{<>}|\geq \epmainthmff{\epsilon}{D}|H_{<>}|$. It is now straightforward to see that (\ref{parameterbounds})-(\ref{importanthypothesis}) hold for all $\eta\in 2^0=\{<>\}$ (note that (\ref{defofset}) and (\ref{importanthypothesis}) are vacuous).  

Suppose now that $0\leq t<d$ and assume we have inductively constructed $\langle H_{\eta}: \eta\in 2^{\leq t}\rangle$, $\langle g_{\eta}: \eta\in 2^{\leq t}\rangle$, $\langle x_{\eta}: \eta\in 2^{\leq t}\rangle$, and $\langle X_{\eta}: \eta\in 2^{\leq t}\rangle$ such that (\ref{parameterbounds})-(\ref{importanthypothesis}) hold for all $\eta \in 2^t$.  We now show how to extend the sequences by defining $H_{\eta\wedge i}, g_{\eta \wedge i}, x_{\eta\wedge i}, X_{\eta\wedge i}$ for each $\eta \in 2^t$ and $i\in \{0,1\}$.  Fix $\eta \in 2^{t}$.  Observe that our induction hypotheses (\ref{ghypp}) and (\ref{largeintersection}) imply that for each $i\in \{0,1\}$,
\begin{align}\label{large1*}
|N^i(g_{\eta})\cap X_{\eta}\cap H_{\eta}|\geq(\epmainthmff{\epsilon}{D}-\epmainthmff{\epsilon}{f^{t-1}(k)})|H_{\eta}|> \epmainthmff{\epsilon}{f^t(k)}|H_{\eta}|,
\end{align}
where the last inequality is a consequence of (\ref{fact000}).  For each $i\in \{0,1\}$, because $A$ is $k$-stable, Lemma \ref{lem:stillstable} implies that $N^i(g_{\eta})=A^i-g_{\eta}$ is $(k+1)$-stable. By our induction hypothesis (\ref{stabilityhypothesis}), $X_{\eta}$ is $f^{t-1}(k)$-stable.  Consequently, for each $i\in \{0,1\}$, $N^i(g_{\eta})\cap X_{\eta}$ is $h(k+1,f^{t-1}(k))=f^{t}(k)$-stable.

Combining the fact that $N^i(g_{\eta})\cap X_{\eta}$ is $f^{t}(k)$-stable with (\ref{large1*}), the fact that $\epsilon<\boundkeylemtoy{f^t(k)}$, and the fact that $|G|\geq \nu(f^t(k),\epsilon, mt,p)$, we see that Proposition \ref{prop:keylemtoy} implies that for each $i\in \{0,1\}$, there is a subspace $H_{\eta \wedge i}$ of $H_{\eta}$ and $x_{\eta \wedge i}\in G$ such that $H_{\eta \wedge i}$ has codimension at most $m$ in $H_{\eta}$ and
\begin{align}
|(N^i(g_{\eta})&\cap X_{\eta}) \cap (H_{\eta \wedge i}+x_{\eta \wedge i})|\geq (1-\epkeylemtoy{\epsilon}{f^t(k)})|H_{\eta \wedge i}|.\label{large7*}
\end{align}
For each $i\in \{0,1\}$, set $X_{\eta \wedge i}:=(N^i(g_{\eta})\cap X_{\eta})-x_{\eta\wedge i}=N^i(g_{\eta}+x_{\eta \wedge i})\cap( X_{\eta}-x_{\eta \wedge i} )$.  Observe that by our inductive hypothesis (\ref{parameterbounds}) on $H_{\eta}$, for each $i\in \{0,1\}$ the codimension of $H_{\eta \wedge i}$ in $G$ is at most $m+mt=m(t+1)$, establishing (\ref{parameterbounds}). Since $m(t+1)\leq md$, $H_{\eta\wedge i}$ cannot be $\epmainthmff{\epsilon}{D}$-good.  Consequently,  there is $g_{\eta\wedge i}\in G$ such that for each $j\in \{0,1\}$, $|N^j(g_{\eta\wedge i})\cap H_{\eta\wedge i}|> \epmainthmff{\epsilon}{D}|H_{\eta\wedge i}|$.  This completes our construction of $H_{\tau}, x_{\tau}, g_{\tau}, X_{\tau}$ for each $\tau \in 2^{t+1}$. 

We now show that conditions (\ref{ghypp})-(\ref{importanthypothesis}) hold for each $\tau\in 2^{t+1}$.  Fix $\tau\in 2^{t+1}$, so $\tau=\eta\wedge i$ for some $i\in \{0,1\}$ and $\eta \in 2^t$.  Observe that (\ref{defofset}) holds for $\tau$ by definition of $X_{\eta \wedge i}$, (\ref{largeintersection}) holds for $\tau$ by (\ref{large7*}), and (\ref{ghypp}) holds for $\tau$ by our choice of $g_{\eta \wedge i}$. For (\ref{stabilityhypothesis}), note that $X_{\tau}=(N^i(g_{\eta})\cap X_{\eta})-x_{\eta\wedge i}$ is just a translate of $N^i(g_{\eta})\cap X_{\eta}$.  We know that $N^i(g_{\eta})\cap X_{\eta}$ is $f^{t}(k)$-stable, so Lemma \ref{lem:stillstable} implies that $X_{\tau}$ is $f^{t}(k)$-stable.  It remains to check condition (\ref{importanthypothesis}).  Suppose that $s< t+1$, $\sigma \in 2^s$ is such that $\sigma \triangleleft \tau$, and $x\in X_{\tau}$.  We want to show that
\begin{align}\label{6*}
x+g_{\sigma}+ x_{\tau|_{s+1}}+\ldots+x_{\tau|_{t}}+x_{\tau} \in A \Leftrightarrow \tau(s+1)=1.
\end{align}
Suppose first that $s=t$, so $\sigma = \eta$ (recall that $\tau=\eta \wedge i$).  Hence, rewriting (\ref{6*}), our aim is to show that $x+g_{\eta}+x_{\tau} \in A \Leftrightarrow \tau(t+1)=1$. Since $\tau(t+1)=i$, this means that we need to show that $x+g_{\eta}+x_{\tau} \in A \Leftrightarrow i=1$.  Now by definition, $X_{\tau}\subseteq N^i(g_{\eta}+x_{\tau})$.  Thus, since $x\in X_{\tau}$, $x\in N^i(g_{\eta}+x_{\tau})$ implies $x+g_{\eta}+x_{\tau}\in A^i$. It follows that $x+g_{\eta}+x_{\tau}\in A$ if and only if $i=1$, as desired.

Suppose now that $s<t$.  Then $\sigma\triangleleft \eta$ and 
\begin{equation}
x+g_{\sigma}+ x_{\tau|_{s+1}}+\ldots+x_{\tau|_{t}}+x_{\tau}=x+g_{\sigma}+x_{\eta|_{s+1}}+\ldots +x_{\eta|_{t-1}}+x_{\eta}+x_{\tau}. \label{9}
\end{equation}
Since $x\in X_{\tau} = N^i(g_{\eta}+x_{\tau})\cap (X_{\eta}-x_{\tau}) $, we see that $x+x_{\tau}\in X_{\tau}+x_{\tau}\subseteq X_{\eta}$.  By our induction hypothesis (\ref{importanthypothesis}),
$$
(x+x_{\tau})+g_{\sigma}+ x_{\eta|_{s+1}}+\ldots+x_{\eta|_{t-1}}+x_{\eta} \in A \Leftrightarrow \eta(s+1)=1,
$$
and thus
$$
x+g_{\sigma}+ x_{\tau|_{s+1}}+\ldots+x_{\tau|_{t}}+x_{\tau} \in A \Leftrightarrow \eta(s+1)=1 \Leftrightarrow \tau(s+1)=1,
$$
where the last equivalence is a consequence of $\tau=\eta \wedge i$ and $s<t$.  This finishes our verification of (\ref{importanthypothesis}) for $\tau$.

We may henceforth assume that we have constructed sequences (a)-(c) satisfying properties (\ref{parameterbounds})-(\ref{importanthypothesis}). We now show that the tree bound $d(\Gamma_A)$ satisfies $d(\Gamma_A)>d$, contradicting the initial assumption that $A$ is $k$-stable.

In order to construct a tree of height $d$, for each $\eta \in 2^d$ choose $c_{\eta}\in X_{\eta}$ and set $a_{\eta}:=c_{\eta}+\sum_{\sigma\trianglelefteq \eta}x_{\sigma}$.  Let $b_{<>}:=g_{<>}$ and for each $0<s<d$ and $\sigma \in 2^{s}$, let $b_{\sigma}:= g_{\sigma}-x_{\sigma|_{1}}-\ldots -x_{\sigma|_{s-1}}-x_{\sigma}$.  Then for all $s<d$, $\sigma\in 2^s$, $\eta\in 2^d$, and $\sigma \triangleleft \eta$, we have
\begin{align*}
a_{\eta}+b_{\sigma}&=c_{\eta}+\sum_{\tau\trianglelefteq \eta}x_{\tau}+g_{\sigma}-x_{\sigma|_{1}}-\ldots -x_{\sigma|_{s-1}}-x_{\sigma}\\
&=c_{\eta}+\sum_{\tau\trianglelefteq \eta}x_{\tau}+g_{\sigma}-x_{\eta|_{1}}-\ldots -x_{\eta|_{s-1}}-x_{\eta|_s}\\
&=c_{\eta}+g_{\sigma}+x_{\eta|_{s+1}}+\ldots +x_{\eta|_{d}}.
\end{align*}
Since $c_{\eta}\in X_{\eta}$, (\ref{5}) with $t=d$ implies that $a_{\eta}+b_{\sigma}\in A$ if and only if $\eta(s+1)=1$.  Thus we have shown that if $\sigma \triangleleft \eta$, then $a_{\eta}+b_{\sigma}\in A$ if and only if $\sigma \wedge 1\trianglelefteq \eta$.  Definition \ref{def:treebound} now implies that $d(\Gamma_A)> d$ as claimed, concluding the proof.
\end{proof}

\section{Concluding remarks}\label{sec:remarks}

With Theorem \ref{thm:mainthmff} in hand, we are now able to give a quick proof of Corollary \ref{cor:ff}, which states that stable sets are closely approximated by a union of cosets of a subspace of bounded codimension.

\begin{proofof}{Corollary \ref{cor:ff}}
Let $n_1=n_0(k,\epsilon,p)$ be as in Theorem \ref{thm:mainthmff}. Let $d=d(k)$ be as in Theorem \ref{thm:treefact} and set $D:=f^d(k)$, where $f=f_k$ is as in Theorem \ref{thm:mainthmff}.  Then set $h(x):=2d(4/x)^{2D}+2d(4)^{D^2}$. Theorem \ref{thm:mainthmff} implies there is an $\epsilon$-good subspace $H$ of $G$ of codimension at most $ h(1/\epsilon)$.  Since $H$ is $\epsilon$-good for $A$, we have that for all $g\in G$, either $|(A-g)\cap H|=|A\cap (H+g)|\leq \epsilon |H|$ or $|(A-g)\cap H|=|A\cap (H+g)|\geq (1-\epsilon) |H|$.  Let $I:=\{g+H\in G/H: |A\cap (H+g)|\geq (1-\epsilon) |H|\}$, and let 
\begin{align*}
X:=\bigcup_{g+H\in I}(g+H)\text{ and }Y:=\bigcup_{g+H\in (G/H)\setminus I}(g+H). 
\end{align*}
Then by definition of $X$, $Y$, $I$ and $\epsilon$-goodness of $H$, we have that 
\[|A\setminus X|=|A\cap Y|\leq (\epsilon |H|)|(G/H)\setminus I|=\epsilon|H|(|G|/|H|-|I|)\]
and
\[|X\setminus A|\leq \epsilon |H||I|. \]
Thus $|A\Delta X|\leq \epsilon|H|(|G|/|H|-|I|)+\epsilon |H||I| = \epsilon |H|(|G|/|H|)=\epsilon |G|$ as desired.
\end{proofof}

For comparison, the Freiman-Ruzsa theorem \cite[Proposition 10.2]{Green:2005ic} states that $A$ can be efficiently covered by cosets of a not too large subspace in the case that $A$ itself has small doubling (see Corollary \ref{cor:cover}). Here, instead, we obtain that $A$ essentially consists of a union of cosets of a subspace that is not too small.

It is natural to ask whether the property of having an efficient regularity lemma is robust, for example with respect to symmetric differences. Our next corollary shows that this is indeed the case.
\begin{corollary}[Close to stable implies efficient regularity]\label{cor:close}
For all $k\geq 2$, $\epsilon>0$ and primes $p$, there is $n_2=n_2(k,\epsilon,p)$ and a polynomial $h(x)$, depending on $k$, such that for all $n\geq n_2$ the following holds.  Let $A,A'\subseteq G:=\F_p^n$ be subsets such that $A'$ is $k$-stable, $|A|/|G|\leq p^{-h(1/\epsilon)}$, and $|A\Delta A'|\leq \epsilon |A|$.  Then there is a subgroup $H\leqslant G$ of codimension at most $h(1/\epsilon)$ which is $3\epsilon$-uniform for $A$ with respect to $H$. 
\end{corollary}
\begin{proof}
Let $n_2=n_0(k,\epsilon,p)$ be as in Theorem \ref{thm:mainthmff} and $h(x)$ as in Corollary \ref{cor:ff}.  Suppose that $n\geq n_2$ and that $A,A'\subseteq G= \mathbb{F}_p^n$ are such that $A'$ is $k$-stable and $|A\Delta A'|\leq \epsilon |A|$.  By Theorem \ref{thm:mainthmff}, there is a subspace $H\leqslant G$ which is $\epsilon$-good for $A'$ with codimension at most $h(1/\epsilon)$.  Since $H$ is $\epsilon$-good for $A'$, we have that every $y\in G$ is $\epsilon$-uniform for $A'$ with respect to $H$.  It is straightforward to check that $|A\Delta A'|\leq \epsilon |A|$ implies that for all $y\in G$ and $t \in \widehat{G}$, $|\widehat{f_{H,A}^y}(t)-\widehat{f_{H,A'}^y}(t)| \leq 2\epsilon |A|/|H|\leq 2\epsilon$, where the latter inequality holds provided that $A$ has density at most $p^{-h(1/\epsilon)}$ in $G$. It follows that $\sup_{t\in \widehat{G}}|\widehat{f_{H,A}^y}(t)|\leq 2\epsilon+\epsilon= 3\epsilon$ as desired.
\end{proof}

This shows that a set $A$ may have an efficient arithmetic regularity lemma while being itself rather unstable: consider $A:=A'\Delta B$, where $A'\subseteq G$ is a $k$-stable set whose size goes to infinity with $|G|$ but is bounded above by $p^{-h(1/\epsilon)}$, and $B$ is chosen at random from $G$ with probability $\epsilon|A'|/|G|$.  It is not difficult to see that there exists some function $r(n)\rightarrow \infty$ as $n\rightarrow \infty$ such that with high probability $A$ has the $r(|G|)$-order property, yet Corollary \ref{cor:close} asserts that it has an efficient regularity lemma. It would be interesting to understand ``how unstable'' a set $A$ must be in order to preclude the possibility of efficient regularity.  

In this context it is worth observing that Example 4 below, given by Green and Sanders \cite{Green:2015wp} to show that the arithmetic regularity lemma must allow for the existence of non-uniform elements, exhibits an instance of the order property that grows roughly as $\log|G|$. Specifically, the set $A\subseteq \mathbb{F}_3^n$ defined in (\ref{eq:badset}) below has the property that for any subspace $V\leqslant\mathbb{F}_3^n$ of positive dimension,
\[\sup_{t\notin V^\perp} |\widehat{f^0_{V,A}}(t)|\geq \frac{\sqrt{3}}{6},\]
in other words, $0$ is always non-uniform for $A$ with respect to $V$.

\begin{example}\label{exa:Ahasoder1}
The set $A\subseteq \F_3^n$ defined by
\begin{equation}\label{eq:badset}
A:=\{x\in \mathbb{F}_3^n: \text{ there exists $i$ such that $x_1=\ldots=x_i=0$ and $x_{i+1}=1\}$}.
\end{equation}
has the $(n-2)$-order property.

To see this, let $e_i$ denote the $i$ standard basis vector as usual.  For each $1\leq i\leq n-2$, set $a_i:=e_{i+1}$ and for each $1\leq j\leq n-2$, set $b_j:=2e_{j+2}+2e_{j+3}+\ldots +2e_{n}$.  Suppose that $i\leq j$.  Then $i+1\leq j+1<j+2$, and therefore
$$
a_i+b_j=e_{i+1}+2e_{j+2}+\ldots +2e_n=(0,\ldots,0,1,\ldots)\in A.
$$
On the other hand, if $i>j$, then $i+1\geq j+2$ and consequently,
\begin{align*}
a_i+b_j&=e_{i+1}+2e_{j+2}+\ldots+2e_{i+1}+\ldots +2e_n=(0,\ldots, 0,2,\ldots,0,\ldots,2)\notin A.
\end{align*}
\end{example}

Similarly, one might want to establish the existence of a large instance of the order-property in the example given by Green \cite{Green:2005kh} and Hosseini, Lovett and Moshkovitz \cite{Hosseini:2014tv}, which shows that the tower-type bound that arises naturally in the proof of the arithmetic regularity lemma is in fact necessary. Due to the technical complexity of this construction we shall refrain from doing so here, but it is not too difficult to convince oneself that it is not $k$-stable for any fixed $k$.

Not surprisingly, Theorem \ref{thm:mainthmff} implies a regularity lemma for the Cayley graph $\Gamma(G,A)$, where $A\subseteq G=\mathbb{F}_p^n$. It is informative to compare its conclusions to those of the stable graph regularity lemma of \cite{Malliaris:2014go}, and we shall do so below.

\begin{corollary}[Stable regularity in the Cayley graph]\label{cor:mainffreformulation}
For all $k\geq 2$, $\epsilon \in (0,1)$, and primes $p$, there is $n_3=n_3(k,\epsilon, p)$ and $C=C(k)\leq \max\{2d(4/\epsilon)^{2D}, 2d(4)^{D^2}\}$ such that the following holds for all $n\geq n_3$.  If $A\subseteq \mathbb{F}_p^n$ is $k$-stable and $\Gamma_A$ is its Cayley graph, then there is a partition $\{V_1,\ldots, V_M\}$ of $V(\Gamma_A)$ such that $M\leq C$, such that the $V_i$ are exactly the cosets of a subspace of $G$, and such that for each $1\leq i,j\leq M$, there is $t(i,j)\in \{0,1\}$ so that for all $z\in V_i$, $|N^{t(i,j)}(z)\cap V_j|\geq (1-\epsilon)|V_j|$.
\end{corollary}
\begin{proof}
Let $d=d(k)$ be as in Theorem \ref{thm:treefact} and set $D:=f^d(k)$ as before. Let $n_3=n_0(k,\epsilon,p)$ as in Theorem \ref{thm:mainthmff}.  Theorem \ref{thm:mainthmff} implies there is an $\epsilon$-good subspace $H$ of $G$ of codimension at most $ \max\{2d(4/\epsilon)^{2D}, 2d(4)^{D^2}\}$. We need to show that the cosets of $H$ form a partition with the desired property.  

Fix two cosets $H+x$ and $H+y$.  Our aim is to obtain $t=t(x,y)\in \{0,1\}$ such that for all $z\in H+x$, $|N^t(x)\cap (H+y)|\leq \epsilon |H|$. Because $H$ is $\epsilon$-good, either $|(A-x-y)\cap H|\leq \epsilon |H|$ or $|H\setminus (A-y-x)|\leq \epsilon |H|$.  Suppose first that $|(A-x-y)\cap H|\leq \epsilon |H|$. Set $t:=1$ and fix $z\in H+x$, so that $z=h+x$ for some $h\in H$.  Then 
$$
|N^t(z)\cap (H+y)|=|(A-z)\cap (H+y)|=|(A-x-h-y)\cap H|= |(A-x-y)\cap H|\leq \epsilon |H|.
$$
If on the other hand, $|H\setminus (A-y-x)|\leq \epsilon |H|$, set $t:=0$ and perform an analogous argument.
\end{proof}

Since the parts in the partition described in Corollary \ref{cor:mainffreformulation} are the cosets of a subspace, they all have the same size.  It is also straightforward to check that, because $H$ is $\epsilon$-good, any pair of cosets $(H+x,H+y)$ is $\epsilon^{1/2}$-regular in $\Gamma(G,A)$ and has density in $[0,\epsilon]\cup [1-\epsilon]$.  Thus we obtain the main properties of the comparable stable graph regularity lemma of \cite[Theorem 5.18]{Malliaris:2014go}, namely the polynomial bound for the number of parts, the absence of irregular pairs, and the pairwise densities being close to $0$ or $1$.  

There are, however, a few differences between Corollary \ref{cor:mainffreformulation} and \cite[Theorem 5.18]{Malliaris:2014go}. First, the bound on the number of parts in \cite[Theorem 5.18]{Malliaris:2014go} is better than that given in Corollary \ref{cor:mainffreformulation}, the latter providing a bound of the form $O(\epsilon^{-2D})$, where $D$ is at least a tower of $2$'s of height $2^k$, while the former yields $O(\epsilon^{-2^k})$. Further, the parts in the regularity lemma of \cite{Malliaris:2014go} have a property called \emph{$\epsilon$-excellence} \cite[Definition 5.2]{Malliaris:2014go}, which is stronger than $\epsilon$-goodness.  On the other hand, the partition in Corollary \ref{cor:mainffreformulation} has two additional structural properties not present in \cite[Theorem 5.18]{Malliaris:2014go}.  First, the parts of the partition are the cosets of a subgroup and not arbitrary subgraphs.  Secondly, in Corollary \ref{cor:mainffreformulation}, we find that for each pair of parts $U, W$, there is $t(U,W)\in \{0,1\}$ such that \emph{for all} $z\in U$, $|N^{t(U,W)}(z)\cap W|\geq (1-\epsilon)|W|$.  In \cite[Theorem 5.18]{Malliaris:2014go} a slightly weaker property is obtained: for each pair of parts $U,W$  there is $t(U,W)\in \{0,1\}$ such that \emph{for all but at most $\epsilon|U|$ many} $z$ in $U$, $|N^{t(U,W)}(z)\cap W|\geq (1-\epsilon)|W|$.  Whether these differences are essential features or mere by-products of the known proofs remains to be investigated.

\bibliography{stablebibi.bib}

\providecommand{\bysame}{\leavevmode\hbox to3em{\hrulefill}\thinspace}
\providecommand{\MR}{\relax\ifhmode\unskip\space\fi MR }
\providecommand{\MRhref}[2]{%
  \href{http://www.ams.org/mathscinet-getitem?mr=#1}{#2}
}
\providecommand{\href}[2]{#2}
\begin{thebibliography}{10}

\bibitem{Alon:2007gu}
Noga Alon, Eldar Fischer, and Ilan Newman, \emph{{Efficient testing of
  bipartite graphs for forbidden induced subgraphs}}, SIAM J. Comput.
  \textbf{37} (2007), no.~3, 959--976.

\bibitem{Alon:2005hj}
Noga Alon, Janos Pach, R~Pinchasi, and R~Radoi{\v c}i{\'c}, \emph{{Crossing
  patterns of semi-algebraic sets}}, J. Combin. Theory Ser. A \textbf{111}
  (2005), no.~2, 310--326.

\bibitem{Baldwin:2017wm}
John~T Baldwin, \emph{{Fundamentals of stability theory}}, Perspectives in
  Mathematical Logic, Springer-Verlag, Berlin, Berlin, Heidelberg, 1988.

\bibitem{Basu:2009cr}
Saugata Basu, \emph{{Combinatorial complexity in o-minimal geometry}},
  Proceedings of the London Mathematical Society \textbf{100} (2009), no.~2,
  405--428.

\bibitem{Chernikov:2015tt}
Artem Chernikov and Sergei Starchenko, \emph{{Regularity lemma for distal
  structures}}, arXiv:1507.01482 (2015).

\bibitem{Chernikov:2016we}
\bysame, \emph{{Definable regularity lemmas for NIP hypergraphs}},
  arXiv:1607.07701 (2016).

\bibitem{Conant:2017uw}
Gabriel Conant, Anand Pillay, and Caroline Terry, \emph{{A group version of
  stable regularity}}, arXiv:1710.06309 (2017).

\bibitem{Conlon:2012es}
David Conlon and Jacob Fox, \emph{{Bounds for graph regularity and removal
  lemmas}}, Geom. Funct. Anal. \textbf{22} (2012), no.~5, 1191--1256.

\bibitem{Fox:2012us}
Jacob Fox, Mikhail Gromov, Vincent Lafforgue, Assaf Naor, and Janos Pach,
  \emph{{Overlap properties of geometric expanders}}, J. Reine Angew. Math.
  \textbf{671} (2012), 49--83.

\bibitem{Gowers:1997fh}
Timothy Gowers, \emph{{Lower bounds of tower type for Szemer{\'e}di's
  uniformity lemma}}, Geom. Funct. Anal. \textbf{7} (1997), no.~2, 322--337.

\bibitem{Gowers:2011et}
Timothy Gowers and Julia Wolf, \emph{{Linear forms and higher-degree uniformity
  for functions on $\mathbb{F}^n_p$}}, Geom. Funct. Anal. \textbf{21} (2011),
  no.~1, 36--69.

\bibitem{Green:2005kh}
Ben Green, \emph{{A Szemer{\'e}di-type regularity lemma in abelian groups, with
  applications}}, Geom. Funct. Anal. \textbf{15} (2005), no.~2, 340--376.

\bibitem{Green:2005ic}
\bysame, \emph{{Finite field models in additive combinatorics}}, Surveys in
  combinatorics 2005 (Bridget~S Webb, ed.), Cambridge Univ. Press, Cambridge,
  Cambridge, 2005, pp.~1--27.

\bibitem{Green:2015wp}
Ben Green and Tom Sanders, \emph{{Fourier uniformity on subspaces}},
  arXiv:1607.07701 (2015).

\bibitem{Green:2010fk}
Ben Green and Terence Tao, \emph{{An arithmetic regularity lemma, an associated
  counting lemma, and applications}}, An Irregular Mind (2010), 261--334.

\bibitem{Wilfrid:1993wx}
Wilfrid Hodges, \emph{{Model theory}}, Encyclopedia of Mathematics and its
  Applications, 1993.

\bibitem{Hosseini:2014tv}
Kaave Hosseini, Shachar Lovett, and Guy Moshkovitz, \emph{{An improved lower
  bound for arithmetic regularity}}, Mathematical Proceedings of the Cambridge
  Philosophical Society \textbf{161} (2016), no.~2, 193--197.

\bibitem{Komlos:2002gf}
Janos Koml{\'o}s, Ali Shokoufandeh, Mikl{\'o}s Simonovits, and Endre
  Szemer{\'e}di, \emph{{The Regularity Lemma and Its Applications in Graph
  Theory}}, Theoretical Aspects of Computer Science, Springer Berlin
  Heidelberg, Berlin, Heidelberg, 2002, pp.~84--112.

\bibitem{Kral:2008wg}
Daniel Kr{\'a}l, Oriol Serra, and Lluis Vena, \emph{{A combinatorial proof of
  the removal lemma for groups}}, J. Combin. Theory Ser. A \textbf{116} (2009),
  no.~4, 971--978 (English).

\bibitem{Lovasz:2010tm}
Laszlo Lovasz and Balazs Szegedy, \emph{{Regularity partitions and the topology
  of graphons}}, An Irregular Mind (2010).

\bibitem{Malliaris:2014go}
Maryanthe Malliaris and Saharon Shelah, \emph{{Regularity lemmas for stable
  graphs}}, Trans. Amer. Math. Soc. \textbf{366} (2014), no.~3, 1551--1585.

\bibitem{Pillay:oRBBGvRc}
Anand Pillay, \emph{{Geometric stability theory}}, Oxford Logic Guides,
  vol.~32, The Clarendon Press, Oxford University Press, New York, 1996.

\bibitem{Poizat:2001vx}
Bruno Poizat, \emph{{Stable Groups, Mathematical Surveys and Monographs Vol.
  87}}, American Mathematical Society, 2001.

\bibitem{Ruzsa:1999uq}
Imre~Z. Ruzsa, \emph{{An analog of Freiman's theorem in groups}},
  Ast{\'e}risque (1999), no.~258, xv, 323--326.

\bibitem{Shelah:1990ti}
Saharon Shelah, \emph{{Classification Theory and the Number of Non-Isomorphic
  Models}}, 2nd edition ed., vol.~92, Studies in Logic and The Foundations of
  Mathematics, Elsevier, 1990.

\bibitem{Szemeredi:1975le}
Endre Szemer{\'e}di, \emph{{On sets of integers containing no $k$ elements in
  arithmetic progression}}, Acta Arith. \textbf{27} (1975), 199--245.

\bibitem{Wolf:2014vy}
Julia Wolf, \emph{{Finite field models in arithmetic combinatorics -- ten years
  on}}, Finite Fields Appl. \textbf{32} (2015), 233--274.

\end{thebibliography}
\bibliographystyle{amsplain}

\end{document}